\documentclass[11pt,a4paper]{article}
\usepackage[margin=1in]{geometry}
\usepackage[T1]{fontenc}
\usepackage[utf8]{inputenc}
\usepackage{tikz}
\usepackage{hyperref}
\usepackage{amsmath,amssymb,amsfonts,amsthm,amscd, bm, mathtools}
\makeatletter
\newtheorem*{rep@theorem}{\rep@title}
\newcommand{\newreptheorem}[2]{%
\newenvironment{rep#1}[1]{%
 \def\rep@title{#2 \ref{##1}}%
 \begin{rep@theorem}}%
 {\end{rep@theorem}}}
\makeatother

\makeatletter
\def\namedlabel#1#2{\begingroup
	#2%
	\def\@currentlabel{#2}%
	\phantomsection\label{#1}\endgroup
}
\makeatother
\newtheorem{thm}{Theorem}[section]
\newreptheorem{theorem}{Theorem}
\newtheorem{prop}[thm]{Proposition}

\newtheorem{lemma}[thm]{Lemma}
\newtheorem{corol}[thm]{Corollary}

\theoremstyle{definition}
\newtheorem{remark}{Remark}[section]
\author{Leandro Aurichi, Paulo Magalhães Júnior and Lucas Real\footnote{Institute of Mathematics and Computer Sciences, University of S\~{a}o Paulo, S\~{a}o Paulo, Brazil}}
\date{}

\begin{document}

\title{Topological remarks on end and edge-end spaces}




\maketitle

\begin{abstract}
	The notion of ends in an infinite graph $G$ might be modified if we consider them as equivalence classes of infinitely edge-connected rays, rather than equivalence classes of infinitely (vertex-)connected ones. This alternative definition yields the \textit{edge-}end space $\Omega_E(G)$ of $G$, in which we can endow a natural (edge-)end topology. For every graph $G$, this paper proves that $\Omega_E(G)$ is homeomorphic to $\Omega(H)$ for some possibly another graph $H$, where $\Omega(H)$ denotes its usual end space. However, we also show that the converse statement does not hold: there is a graph $H$ such that $\Omega(H)$ is not homeomorphic to $\Omega_E(G)$ for any other graph $G$. In other words, as a main result, we conclude that the class of topological spaces $\Omega_E = \{\Omega_E(G) : G \text{ graph}\}$ is strictly contained in $\Omega = \{\Omega(H) : H \text{ graph}\}$.
\end{abstract}


\section{Introduction:}
\paragraph{}
In a graph $G$, a \textbf{ray} is an one-way infinite path, whose infinite connected subgraphs are called its \textbf{tails}. Similarly, a two-way infinite path is referred to as a \textbf{double-ray}. Over the set $\mathcal{R}(G)$ of rays of $G$, one can define the two (similar) equivalence relations below:

\begin{itemize}
	\item \textbf{End relation ($\sim$):} For $r,s \in \mathcal{R}(G)$, we write $r\sim s$ whenever $r$ and $s$ are infinitely connected, i.e., there are infinitely many (vertex-)disjoint paths connecting $r$ and $s$. Equivalently, no finite set $S\subset V(G)$ \textbf{separates} $r$ and $s$, in the sense that the tails of $r$ and $s$ belong to different connected components of $G\setminus S$. Then, $\sim$ is indeed an equivalence relation on $\mathcal{R}(G)$, whose equivalence class of a ray $r$ (written as $[r]$) is called an \textbf{end} of $G$. The quotient $\mathcal{R}(G)/\sim$, referred to as the \textbf{end space} of $G$, is denoted by $\Omega(G)$;
	\item \textbf{Edge-end relation ($\sim_E$):} Now, for $r,s \in \mathcal{R}(G)$, we write $r\sim_E s$ whenever $r$ and $s$ are infinitely edge-connected, i.e., there is an infinite family of \textit{edge-}disjoint paths connecting infinitely many vertices of $r$ to infinitely many vertices of $s$. Equivalently, no finite set $F\subset E(G)$ \textbf{separates} $r$ and $s$, in the sense that the tails of $r$ and $s$ belong to different connected components of $G\setminus F$. Under $\sim_E$, the equivalence class of a ray $r$ is denoted by $[r]_E$ and is called its \textbf{edge-end}. Then, we fix the notation $\Omega_E(G) = \mathcal{R}(G)/\sim_E$ for the \textbf{edge-end} space of $G$. 
\end{itemize}

Roughly speaking, $\sim_E$ is obtained from $\sim$ after changing the roles of vertices by edges. In particular, $\sim_E$ is a weaker identification, since infinitely (vertex-)connected rays are also infinitely edge-connected. If $G$ is locally finite, then both equivalence relations turn out to be the same. However, when there are vertices of infinite degree, $\sim_E$ might identify more rays. For example, the graph presented by Figure \ref{raioduplodominado} has two ends but only one edge-end.

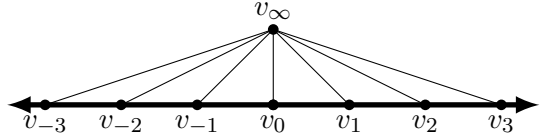
\begin{figure}[ht]
	\centering
	\begin{tikzpicture}
		\draw[line width=2pt, black, -latex] (0,0) -- (3.5,0);
		\draw[line width=2pt, black, -latex] (0,0) -- (-3.5,0);
		
		\foreach \x in {-3,-2,-1,0,1,2,3}
		\fill (\x,0) circle (2pt) node[below]{$v_{\x}$};
		
		\fill (0,1) circle (2pt) node[above]{$v_{\infty}$};
		
		\foreach \x in {-3,-2,-1,0,1,2,3}
		\draw (\x,0) -- (0,1);
		
		\fill (0,0) circle (2pt);
	\end{tikzpicture}
	
	\caption{This graph has two ends, corresponding to the rays $v_0v_1v_2\dots$ and $v_0v_{-1}v_{-2}\dots$, for example. However, it has only one edge-end.}
	\label{raioduplodominado}
\end{figure}

In the literature, the definition of $\sim$ as just presented was first introduced by Halin in \cite{halin}, although the notion of ends can be identified in previous algebraic discussions carried out by Hopf in \cite{hopf} and by Freudenthal in \cite{freudenthal} as standard references within geometric group theory. Since then, a sort of results concerning finite graphs was generalized for infinite ones with the support of $\Omega(G)$. Regarding locally finite graphs, for example, the reader might consult the survey of Diestel in \cite{survey}.

In its turn, inspired by that previous work of Halin, the above formalization of $\sim_E$ is due to Hahn, Laviolette and \v{S}irá\v{n} in \cite{edgeends}. When describing completions of metrics induced by edge lengths, Georgakopoulos in \cite{Georgakopoulos} also relied on this edge-end definition for introducing the $\mathrm{ETOP}$ topology over an infinite graph $G$ and its structure as $1$-complex. Besides these papers, few other references in the literature mention the edge-end structure of infinite graphs, although it is naturally suggested by Halin's original definition when considering an edge-related viewpoint. This incomplete state of art may be justified by the coincidence of $\sim$ and $\sim_E$ when restricted to locally finite graphs, so that the latter relation can be omitted in several studies within infinite graph theory. Despite that, the preprint \cite{nosso} exemplifies how edge-ends might be suitable for discussing edge-connectivity results in arbitrary infinite graphs, even allowing, for instance, extensions of Menger-type theorems.

In any case, when introducing the quotient spaces $\Omega(G)$ and $\Omega_E(G)$, the corresponding topologies below arise from the separating notions displayed by vertices and edges: 

\begin{itemize}
	\item \textbf{End space topology:} For an end $[r]\in \Omega(G)$, we fix a finite set $S\subset V(G)$ to define the following open basic neighborhood: $$\Omega(S,[r]) = \{[s]\in \Omega(G) : S\text{ does not separate }r\text{ and }s\}.$$
	\item \textbf{Edge-end space topology:} For an edge-end $[r]_E \in \Omega_E(G)$, we fix a finite set $F\subset E(G)$ to define the following open basic neighborhood: $$\Omega_E(F,[r]_E) = \{[s]_E\in \Omega_E(G) : F\text{ does not separate }r\text{ and }s\}.$$
\end{itemize}

By highlighting topological properties of the above definitions, this paper aims to investigate their differences even if the underlying graphs are changed. More precisely, let us denote by $\Omega$ and $\Omega_E$ the classes of topological spaces that arise, respectively, as end spaces and edge-end spaces of graphs. By considering disjoint union of $|X|$ rays, we can quickly check that $\Omega$ and $\Omega_E$ contain every discrete space $X$, for example. Actually, $\Omega$ and $\Omega_E$ also contains all the complete ultrametric spaces, which are the end spaces of trees (see Theorem 3.1 of \cite{caracterizacao}). In fact, we will describe in Section \ref{representacaoarestas} how every edge-end space can be obtained as the end space of a (possibly another) graph, proving that $\Omega_E\subseteq \Omega$. On a more challenging direction, the studies carried out by Section \ref{contraexemplo} conclude that the Alexandroff duplicate of the Cantor set is not an edge-end space of any graph, although Example 2.6 of \cite{caracterizacao} shows that this topological space belongs to $\Omega$. To summarize, Sections \ref{representacaoarestas}-\ref{contraexemplo} prove the following main result:

\begin{thm}\label{t1}
	$\Omega_E$ is a proper subclass of $\Omega$, i.e., every edge-end space is the end space of some graph, but the converse does not hold. 
\end{thm}

Multiple recent studies regarding end spaces of infinite graphs are revisited in order to conclude the above statement. In particular, Pitz in \cite{caracterizacao} brings a characterization of these spaces via purely topological criteria, answering a long-standing question proposed by Diestel in \cite{diestel1992}. As a correlated reference, due to Kurkofka and Pitz in \cite{representacao}, their representation theorem for end spaces plays a central role in the discussion carried out by Section \ref{contraexemplo}.

\section{Edge-end spaces as certain end spaces}\label{representacaoarestas}

\paragraph{}
In this section we prove the first part of Theorem \ref{t1}, by including the class of edge-end spaces into the class of end spaces. To that aim, it is useful to identify when a vertex is infinitely connected to a ray. Hence, we say that $v \in V(G)$ \textbf{dominates} a ray $r$ if, for every finite set $S\subset V(G)\setminus \{v\}$, the tail of $r$ lies on the same connected component of $v$ in $G\setminus S$. Clearly, this is equivalent to the existence of an infinite family of paths connecting $v$ and $r$ that are disjoint unless by $v$. Therefore, $v$ dominates $r$ if, and only if, it dominates any ray equivalent to $r$, allowing us to say that $v$ dominates the end $[r]$. On the other hand, if $v$ and a tail of $r$ belong to different connected components of $G\setminus S$, we say that the finite set $S\subset V(G)\setminus S$ \textbf{separates} $v$ and $r$, as well as $v$ from $[r]$.

Rewriting the above definitions in an edge-related viewpoint, we say that $v\in V(G)$ \textbf{edge-dominates} the ray $r$ if, for every finite set $F\subset E(G)$, the vertex $v$ and a tail of $r$ belong to the same connected component of $G\setminus F$. Equivalently, there is an infinite family of edge-disjoint paths connecting $v$ and to infinitely many vertices of $r$. Hence, $v$ edge-dominates $r$ if, and only if, it dominates any other representative of the edge-end $[r]_E$. In particular, $v$ also edge-dominates every representative from $[r]$, in which case we say that $v$ \textit{edge-}dominates the \textit{end} $[r]$.

On the other hand, if $v$ and a tail of $r$ belong to different connected components of $G\setminus F$, we say that some finite set $F\subset E(G)$ \textbf{separates} $v$ and $r$, or even $v$ from $[r]_E$. Within these definitions, the fact that $\Omega_E\subseteq \Omega$ follows by an argument of connectedness:

\begin{thm}\label{volta}
	Let $G$ be a graph. Then, there is a graph $H$ whose vertices edge-dominate at most one end (as defined by the relation $\sim$) and such that $\Omega(H) \cong \Omega_E(G)$. 
\end{thm}

Throughout this section, we will prove the above result by considering $H$ a graph obtained from $G$ after expanding some vertices to cliques. Formally, for each vertex $v\in V(G)$ that edge-dominates a ray, let $K_v$ be a complete graph of order $d(v)$, the degree of $v$. Hence, it is possible to fix $\rho_v: N(v) \to V(K_v)$ a bijection between the neighborhood $N(v) := \{u \in V(G) : uv \in E(G)\}$ and the vertices of $K_v$. Denoting by $D$ the set of vertices of $G$ that edge-dominate some ray, the vertex set of $H$ may be given by $V(H) = (V(G)\setminus D) \cup\displaystyle \bigcup_{v\in D}V(K_v) $, so that the elements of the (disjoint) union $\displaystyle \bigcup_{v\in D}V(K_v)$ are called \textbf{expanded vertices}. The edge set of $H$, in its turn, is given by the (disjoint) union $\displaystyle \bigcup_{v \in D}E(K_v)$ and the range of the function $\rho : E(G)\hookrightarrow E(H)$ defined as follows:
\begin{itemize}
	\item $\rho(uv) = uv$ if $u,v \in V(G)\setminus D$;
	\item $\rho(uv) = u\rho_v(u)$ if $u\in V(G)\setminus D$ but $v \in D$;
	\item $\rho(uv) = \rho_u(v)\rho_v(u)$ if $u,v \in D$.
\end{itemize}

Roughly speaking, $\rho$ is a map that attaches different edges incident in a vertex $v\in D$ to different vertices of $K_v$. Since adjacencies between vertices of $V(G)\setminus D$ are preserved, we regard $\rho$ as an \textbf{inclusion map} from $E(G)$ to $E(H)$. For a vertex $x \in K_v$, hence, we define its \textbf{canonical edge} to be the unique edge of $\rho(E(G))$ that has $x$ as an endpoint.

This inclusion map also translates rays of $G$ into rays of $H$ in a natural sense. In fact, if $r = v_0v_1v_2v_3\dots$ is a ray in $G$, we consider the ray $\rho(r)$ in $H$ whose presentation by its \textit{edges} is $\rho(v_0v_1)s_1\rho(v_1v_2)s_2\rho(v_2v_3)\dots$, where

\begin{itemize}
	\item $s_i =\emptyset$ is the empty edge if $v_i\notin D$, for $i \geq 1$;
	\item $s_i  = \rho_{v_i}(v_{i-1})\rho_{v_i}(v_{i+1})$ is the edge in $K_{v_i}$ connecting the canonical edges $\rho(v_{i-1}v_i)$ and $\rho(v_iv_{i+1})$, for $i \geq 1$.
\end{itemize}

Clearly, finite paths of $G$ can be recovered in $H$ by the same construction. In other words, if $P = v_0v_1\dots v_n$ is a path in $G$, we denote by $\rho(P)$ the path in $H$ whose edges are $\rho(v_0v_1)s_1\rho(v_1v_2)s_2\rho(v_2v_3)\dots s_{n-1}\rho(v_{n-1}v_n)$. Conversely, a path or a ray $r$ in $H$ has the form $\rho(s)$ for some path or some ray $s$ of $G$, accordingly, whenever $|E(r) \cap E(K_v)| \leq 1$ for every $v \in D$. In this sense, via $\rho$, the next technical result allows us to map edge-disjoint families of paths in $G$ to disjoint families of paths in $H$:

\begin{lemma}\label{familiadisjunta}
	Fix $r$ a ray in $G$. Let $R$ be a vertex or a ray that cannot be separated from $r$ by finitely many edges. Then, there is an infinite family $\{P_n\}_{n \in \mathbb{N}}$ of edge-disjoint paths such that:
	\begin{enumerate}
		\item For every $n \in \mathbb{N}$, $P_n$ connects $r$ and $R$, i.e., it has one endpoint in $r$ and the other in $R$;
		\item If $n \neq m$, then $V(P_n)\cap V(P_m) \subset D$.
	\end{enumerate}
\end{lemma}
\begin{proof}
	We first remark that, if $R$ is a vertex, then our main hypothesis actually says that $R$ dominates $r$. In this case, if another vertex $v \in V(G)$ cannot be separated from $R$ by finitely many edges, then $v \in D$ as well. This is because $v$, $R$ and a tail of $r$ belong to the same connected component of $G\setminus F$, for every finite set $F\subset E(G)$.

	Then, let $P_0$ be a path connecting $r$ and $s$. For instance, for some $n \geq 1$ suppose that we have so far defined finitely many edge-disjoint paths $P_0,P_1,\dots,P_{n-1}$ that connect $r$ and $R$. Moreover, we assume by induction that $V(P_i)\cap V(P_j)\subset D$ whenever $i\neq j$. But, by the observation made in the previous paragraph, any vertex from $V(G)\setminus D$ can be separated from $r$ and $R$ by a finite set of edges. In particular, there is a finite set $F\subset E(G)$ that separates every vertex of $\displaystyle \bigcup_{i=0}^{n-1}V(P_i)\setminus D$ from $r$ and $R$. By hypothesis, there is a path $P_n$ in $G\setminus F$ that connects the tail of $r$ to $R$ (if it is a vertex) or to its tail (if it is a ray). Moreover, $V(P_n)\cap V(P_i)\subset D$ for every $i < n$ by the choice of $F$.

	At the end of this recursive process, $\{P_n\}_{n \in \mathbb{N}}$ is the claimed family of paths.
\end{proof}

In particular, if $R = s$ is a ray that is edge-equivalent to $r$, then $\{P_n\}_{n\in\mathbb{N}}$ as in the above statement is a family of edge-disjoint paths connecting $r$ and $s$ whose elements intersect possibly at dominating vertices. Hence, in $H$, the family $\{\rho(P_n)\}_{n\in\mathbb{N}}$ turns out to be a family of vertex-disjoint paths connecting $\rho(r)$ to $\rho(s)$, since every expanded vertex is an endpoint of precisely one edge. This means that the map 

$$ \begin{array}{cccc}
	\Phi:  &  \Omega_E(G) & \to & \Omega(H)\\
	& [r]_E & \mapsto & [\rho(r)]
\end{array} $$
is well-defined. Proving Theorem \ref{volta}, our aim now is to verify that $\Phi$ is an homeomorphism. First, relying on Lemma \ref{familiadisjunta}, $\Psi$ is easily seen to be onto:

\begin{prop}\label{sobrejetora}
	The map $\Phi$ is surjective.
\end{prop}
\begin{proof}
	Fix $r = x_0x_1x_2\dots$ a ray in $H$. First, consider the case in which $V(r)\cap V(K_v)$ is finite for every $v \in D$. We will recursively define a ray $r' = x_0'x_1'x_2'\dots$ of $H$ as follows:
	\begin{itemize}
		\item We declare $x_0' = x_0$. If $x_0\in V(K_v)$ for some $v\in D$, we also define $x_1' = x_{i_1}$, in which $i_1 = \max\{j\in\mathbb{N} : x_j \in V(K_v)\}$. Since $K_v$ is a complete graph, $x_0'$ and $x_1'$ are indeed neighbors;
		\item We suppose that $x_0',x_1',x_2',\dots,x_n'$ are already defined for some $n \in\mathbb{N}$. By induction, we can write $x_k' = x_{i_k}$ for every $1 \leq k \leq n$, for certain indices $i_1 < i_2 < i_3 < \dots < i_n$. Moreover, we assume that, if $x_n' \in V(K_v)$ for some $v \in D$, then $i_n = \max\{j\in\mathbb{N} : x_j \in V(K_v)\}$. We hence set $x_{n+1}' = x_{i_n+1}$. In addition, if $x_{i_n+1}\in V(K_v)$ for some $v \in D$, we also set $x_{n+2}' = x_{i_{n+2}}$, where $i_{n+2} = \max\{j\in\mathbb{N} : x_j \in V(K_v)\}$. 
	\end{itemize}
	Defined this way, the set of vertices of $r'$ is actually an (infinite) subset of $V(r)$, from where we have $[r'] = [r]$. By construction, $|E(r) \cap E(K_v)| \leq 1$ for every $v \in D$, so that $r' = \rho(s)$ for some ray $s$ in $G$. Therefore, $\Phi([s]_E) = [r'] = [r]$.

	Now, suppose that $V(r) \cap V(K_v)$ is infinite for some $v \in D$. By definition of $D$, there is $s$ a ray in $G$ that $v$ edge-dominates. Then, there is $\{P_n\}_{n\in \mathbb{N}}$ an infinite family of edge-disjoint paths connecting $v$ and $s$. By Lemma \ref{familiadisjunta}, we can assume that $P_n \cap P_m \subset D$ if $n \neq m$. Since expanded vertices are endpoints of precisely one canonical edge, $\{\rho(P_n)\}_{n\in\mathbb{N}}$ is a family of vertex-disjoint paths that verifies the equivalence between $\rho(s)$ and a ray $r_v$ composed by expanded vertices of $K_v$. On the other hand, since $V(r) \cap V(K_v)$ is infinite and $K_v$ is a complete graph, we have $r \sim r_v$. Hence, $\Phi([s]_E) = [\rho(s)] = [r_v] = [r]$.   
	
\end{proof}

When showing how $\Phi$ is injective, we can actually conclude that distinct edge-ends of $G$ are mapped to ends of $H$ that are not infinitely edge-connected: 

\begin{prop}\label{injecao}
	If $r$ and $s$ are rays of $G$ such that $[r]_E \neq [s]_E$, then $\rho(r)$ and $\rho(s)$ are not edge-equivalent in $H$. In particular, $[\rho(r)]\neq [\rho(s)]$.  
\end{prop}
\begin{proof}
	Fix $r$ and $s$ two rays that can be separated by a finite set $F\subset E(G)$. In particular, $\rho(F)$ is also finite. For instance, suppose that, in $H\setminus \rho(F)$, there is a path $P = x_0x_1\dots x_n$ connecting the tails of $\rho(s)$ and $\rho(r)$. Suppose that $|V(P)|$ is minimum with that property. Then, given distinct $x_i,x_j \in K_v$ for some $v\in D$, we must have $j = i+1$ if $i < j$, because the subpath $P' = x_0x_1\dots x_ix_j\dots x_n$ is well defined and also connects $\rho(s)$ to $\rho(r)$. Hence, if $x_ix_{i+1}\in E(K_v)$, the edges $x_{i-1}x_i$ and $x_{i+1}x_{i+2}$, if exist, are canonical. In other words, $P = \rho(Q)$ for some path $Q$ that connects $r$ and $s$ in $G$. However, this contradicts the fact that $F$ separates the rays $r$ and $s$, because $E(P)\subset E(H)\setminus \rho(F)$ and, thus, $E(Q)\subset E(G)\setminus F$. Therefore, $\rho(F)$ separates $\rho(r)$ and $\rho(s)$.  
\end{proof}

\begin{corol}
	Every vertex of $H$ edge-dominates at most one end of $\Omega(H)$.
\end{corol}
\begin{proof}
	Fix $v \in V(H)$ a vertex that edge-dominates two distinct ends of $H$. Fix $r$ and $s$ representatives for those ends. In particular, if $F\subset E(H)$ is finite, there is $P$ a path in $H\setminus F$ connecting $v$ to a tail of $r$. Then, $F' = F\cup E(P)$ is also a finite set of edges, so that there is a path $P'$ in $H\setminus F'$ connecting $v$ to a tail of $s$. Therefore, by concatenating $P$ and $P'$, we obtain a path in $H\setminus F$ connecting the tails of $r$ and $s$. This proves that $r$ and $s$ are infinitely edge-connected, although belong to different ends of $H$. Since $\Phi$ is surjective, this contradicts Proposition \ref{injecao}.
\end{proof}

Proving Theorem \ref{volta}, then, we finish this section by showing that $\Phi$ is open and continuous:

\begin{prop}\label{phi}
	$\Phi$ is an homeomorphism.
\end{prop}
\begin{proof}
	We will first verify that $\Phi$ is continuous. To this aim, let $\Omega(S,[r])$ be a basic open set in the end space of $H$, for some $[r]\in \Omega(H)$ and some finite $S\subset V(H)$. Since $\Phi$ is a bijection, we can write the representative $r$ as $r = \rho(s)$ for some ray $s$ in $G$. If a fixed vertex $u\in S$ belongs to $D$, let $F_ u = \{f_u\}$ denote the singleton set containing its canonical edge $f_u$. If not, $u$ does not edge-dominate the ray $s$, so that there is a finite set $F_u\subset E(G)$ that separates $s$ and $u$. Hence, $F = \displaystyle \bigcup_{u\in S}F_u$ is a finite set of edges. We then claim that the basic open set $\Omega_E(F,[s]_E)$ for the edge-end space of $G$ is contained in $\Phi^{-1}(\Omega(S,[r]))$. In fact, if $F$ does not separate $s$ and a ray $s'$, there is $P$ a path connecting $s$ and $s'$ in $G\setminus F$. By the choice of $F$, therefore, $\rho(P)$ is a path connecting $\rho(s)$ and $\rho(s')$ in $G\setminus S$. In other words, $S$ does not separate $\rho(s)$ and $\rho(s')$, proving that $\Phi(\Omega_E(F,[s]_E))\subset \Omega(S,[r])$.

	Conversely, in order to show that $\Phi$ is an open map, let $\Omega_E(F,[s]_E)$ be a basic open set containing an edge-end $[s]_E \in \Omega_E(G)$, for some finite $F\subset E(G)$. Hence, it is also finite the set $S = \{x\in V(H) : x \text{ is endpoint of }\rho(e) \text{ for some }e \in F\}$. Then, it is enough to verify that $\Omega(S,[\rho(s)])\subset \Phi(\Omega(F,[s]_E))$. To this aim, again by the fact that $\Phi$ is surjective, an element $[r']\in \Omega(S,[\rho(s)])$ has a representative of the form $r' = \rho(s')$ for some ray $s'$ of $G$. Since $S$ does not separate $\rho(s')$ and $\rho(s)$, there is a path $P$ connecting the tails of these two rays in $H\setminus S$. As in Proposition \ref{injecao}, if we consider $P$ to have as few vertices as possible, we can write $P = \rho(Q)$ for some path $Q$ in $G\setminus F$ connecting $s$ and $s'$. Hence, $F$ does not separate $s$ and $s'$, so that $[s']_E\in \Omega_E(F,[s]_E)$ and, therefore, $[r'] = [\rho(s')]\in \Phi(\Omega_E(F,[s]_E))$.
\end{proof}

\section{Construction of edge-end spaces}\label{sec:construction}

\paragraph{}

Throughout this section, we will prove the converse of Theorem \ref{volta}. Although this is not needed to show that $\Omega_E$ is a proper subclass of $\Omega$, Theorem \ref{volta} and the below result combined give a graph-theoretic description of edge-end spaces:

\begin{thm}\label{ida}
	Let $G$ be a graph in which every vertex edge-dominates at most one end of $\Omega(G)$. Then, there is a graph $H$ such that $\Omega_E(H)\simeq \Omega(G)$.
\end{thm}

We start the construction of $H$ by \textit{enveloping} pre-described sets of ends of $G$, a procedure recently developed by Kurkofka and Pitz in \cite{representacao}. In our setting, however, we will need only the restricted version below:

\begin{lemma}[\cite{representacao}, Theorem 3.2]\label{propriedadesenvelope}
	Let $\varepsilon \in \Omega(G)$ be an end of $G$ and fix $\mathcal{R}_{\varepsilon}\subset \varepsilon$ a maximal family of pairwise (vertex-)disjoint rays. Denote by $D_{\varepsilon}$ the set of vertices that dominate $\varepsilon$. Then, 
	\begin{equation}\label{envelope}
		E_{\varepsilon} = D_{\varepsilon} \cup \displaystyle \bigcup_{r \in \mathcal{R}_{\varepsilon}}V(r)
	\end{equation} 
	has \textit{finite adesion}, i.e., for each connected component $C$ of $G\setminus E_{\varepsilon}$, the set $N(C) = \{v \in E_{\omega}: v \text{ has a neighbor in }C\}$ is finite. Moreover, if $r$ is a ray such that $V(r)\cap E_{\varepsilon}$ is infinite, then $[r] = \varepsilon$.
\end{lemma}

\begin{proof}[Revisited proof from \cite{representacao}]
	Fix a connected component $C$ of $G\setminus E_{\varepsilon}$. Let $\hat{C}$ be a subgraph of $G$ containing $C$, $N(C)$ and precisely one edge connecting each vertex of $N(C)$ to (some vertex of) $C$. In particular, $\hat{C}$ is also a connected subgraph of $G$. For instance, suppose that $N(C)$ is infinite, so that, by the Star-Comb Lemma (see \cite{livrodiestel}), one of the following assertions holds:
	\begin{itemize}
		\item There is a \textit{star} in $\hat{C}$ centered at some vertex $v\in V(C)$ with teethes in $N(C)$. In other words, for every finite set $S\subset V(G)\setminus \{v\}$, there is a path in $\hat{C}\setminus S$ connecting $v$ to a vertex of $N(C)$;
		\item There is a \textit{comb} in $\hat{C}$ with teethes in $N(C)$. More precisely, there is a ray $s$ in $C$ such that, for every finite set $S\subset V(G)$, there is a path connecting a vertex of $N(C)$ to the tail of $s$ in $\hat{C}\setminus S$.
	\end{itemize}

	Suppose that the former case occurs and let $v \in V(C)$ be the center of the mentioned star. Fix a finite set $S\subset V(G)\setminus \{v\}$. Since the family $\mathcal{R}_{\varepsilon}$ is composed by pairwise disjoint rays, $S$ meets only finitely many of them. Hence, by adding initial segments of these rays, there is a finite set $S'\subset V(G)\setminus \{v\}$ containing $S$ such that $r\setminus S'$ is the tail of $r$ for every $r\in \mathcal{R}_{\varepsilon}$. As in the first item above, however, there is in $V(G)\setminus S'$ a path $P$ connecting $v$ to a vertex $u$ of $N(C)$. If $u \in V(r)$ for some $r\in \mathcal{R}_{\varepsilon}$, then $u$ belongs to a tail of $r$ by our choice of $S'$. If $u\in D_{\varepsilon}$, there is a path $P$ in $G\setminus S'$ connecting $u$ to a tail of a ray whose end is $\varepsilon$, since  $u$ dominates $\varepsilon$. In any case, $S$ does not separate $v$ and $\varepsilon$, contradicting the fact that $D_{\varepsilon}\cap V(C) = \emptyset$.

	Suppose now that there is a ray $s$ in $C$ as in the second item above. Again, let $S\subset V(G)$ be any finite set and, as argued in the previous paragraph, assume that $r\setminus S$ is the tail of $r$ in $G\setminus S$ for every $r\in \mathcal{R}_{\varepsilon}$. By our main hypothesis over $s$, there is a path in $G\setminus S$ that connects its tail to a vertex $u\in N(C)$. If $u\in r$ to some $r\in \mathcal{R}_{\varepsilon}$, then $u$ belongs to a tail of $s$. If not, $u \in D_{\varepsilon}$, so that there is a path in $G\setminus S$ connecting $u$ to a tail of some ray whose end is $\varepsilon$. In any case, $S$ does not separate $s$ from $\varepsilon$, so that $[s] = \varepsilon$. However, since $s$ is contained in $C$, this contradicts the maximality of the family $\mathcal{R}_{\varepsilon}$.

	Finally, let $r$ be a ray such that $V(r)\cap E_{\varepsilon}$ is infinite. Then, if $S\subset V(G)$ is finite, there is $v\in V(r)\cap E_{\varepsilon}\setminus S$. If $v$ dominates $\varepsilon$, then there is a path in $G\setminus S$ connecting $v$ to $r'$ for some $r'\in \varepsilon$. If not, then $v$ itself belongs to a ray $r'\in \varepsilon$. In any case, $S$ does not separate $r$ from representatives of $\varepsilon$, proving that $[r] = \varepsilon$.
\end{proof}

Then, for each end $\varepsilon\in \Omega(G)$, we will fix $E_{\varepsilon}$ as defined by (\ref{envelope}) and will call it an \textbf{envelope} for $\varepsilon$. By our main hypothesis over $G$, each vertex $v \in V(G)$ edge-dominates at most one such end $\varepsilon \in \Omega(G)$. If this is the case, we define a new vertex $v'$ and a bipartition $\tau_v : N(v) \to \{v,v'\}$ according to the following rules:
\begin{equation*}
	\tau_v(u) = \left\{\begin{tabular}{ll}
		$v$, & if $u\in E_{\varepsilon}$; \\
		$v'$, & if $u\in V(G)\setminus E_{\varepsilon}$.
	\end{tabular}\right.
\end{equation*}

From now on in this section, $D\subset V(G)$ will denote the set of vertices that edge-dominates an end of $G$. Then, we consider $V(G)\cup \{v' : v \in D\}$ as the vertex set of the claimed graph $H$. Its edge set is $\{vv': v \in D\}\cup \tau(E(G))$, in which $\tau : E(G)\hookrightarrow E(H)$ is the injective map given by:

\begin{equation*}
	\tau(uv) = \left\{\begin{tabular}{ll}
		$uv$, & if $u,v\in V(G)\setminus D$; \\
		$u\tau_v(u)$, & if $u\in V(G)\setminus D$ and $v \in D$; \\
		$\tau_u(v)\tau_v(u)$, & if $u,v\in D$.
	\end{tabular}\right.
\end{equation*}

In other words, $H$ is built from $G$ after duplicating vertices of $D$, rearranging the edges of $G$ according to $\tau$ and defining an edge between a vertex and its copy. For instance, $G$ may be re-obtained from $H$ after the contraction of each edge from $\{vv': v \in D\}$ to a single vertex again. More precisely, the projection map $\pi : V(H) \to V(G)$ given by $\pi(v)=v$ for every $v\in V(G)$ and $\pi(v')=v$ for every $v\in D$ is clearly a bijection verifying $\pi(u)\pi(v)\in E(G)$ or $\pi(u)=\pi(v)$ whenever $uv$ is an edge of $H$. In particular, $K:=(\pi(K'),\tau^{-1}(E(K')))$ defines a connected subgraph of $G$ whenever $K'$ is a connected set of vertices in $H$, from where we highlight the following observation:

\begin{remark}\label{remark33}
    Fix a connected subgraph $K'$ of $H$ and let $K$ denote the subgraph of $G$ described by the edges of $\tau^{-1}(E(K'))$ and their end vertices. Since $\pi(K')$ comprises precisely these endpoints, $K$ is connected.
\end{remark}

Similarly, by tracking the edges of $G$, the map $\tau$ also allows us to include paths and rays of $G$ into $H$. For example, given a ray $r = v_0v_1v_2\dots $ in $G$, we denote by $\tau(r)$ the ray in $H$ whose presentation by its \textit{edges} is $\tau(v_0v_1)s_1\tau(v_1v_2)s_2\tau(v_2v_3)s_3\dots $, in which, for $i \geq 1$:
\begin{itemize}
	\item $s_i = \emptyset$ is the empty edge if $v_i\notin D$ or $v_i\in D$ and $\tau_{v_{i}}(v_{i+1}) = \tau_{v_i}(v_{i-1})$;
	\item $s_i = v_iv_i'$ if $v_i \in D$ and $\tau_{v_i}(v_{i+1})\neq \tau_{v_i}(v_{i-1})$.
\end{itemize}

Regarding the above notation, a path $P = v_0v_1v_2\dots v_n$ in $G$ can also be seen within $H$, via the path $\tau(P)$ whose edges are given by $$\tau(v_0v_1)s_1\tau(v_1v_2)s_2\tau(v_2v_3)s_3\dots s_{n-1}\tau(v_{n-1}v_n).$$ In particular, if $\{P_n\}_{n \in\mathbb{N}}$ is an infinite family of (vertex-)disjoint paths connecting the rays $r$ and $s$ in $G$, then $\{\tau(P_n)\}_{n \in \mathbb{N}}$ are (vertex-)disjoint paths connecting  $\rho(r)$ and $\rho(s)$. Therefore, the map below is well-defined and it is a natural candidate to be an homeomorphism between $\Omega(G)$ and $\Omega_E(H)$:

$$ \begin{array}{cccc}
	\Psi:  &  \Omega(G) & \to & \Omega_E(H)\\
	& [r] & \mapsto & [\tau(r)]_E
\end{array} $$

The fact that $\Psi$ is surjective can be seen as an application of König's Lemma, while $\Psi$ is injective due to the main hypothesis over $G$:

\begin{prop}
	$\Psi$ is bijective.
\end{prop}
\begin{proof}
	Let $r'$ be a ray in $H$, whose presentation via edges might be written as $f_0f_1f_2\dots$. Since the edges $\{vv': v \in D\}\subset E(H)$ are pairwise non-adjacent, $\tau^{-1}(E(r'))$ is an infinite set of edges in $G$. Consider $K$ the subgraph of $G$ that contains precisely these edges and its endpoints, which is connected by Remark \ref{remark33}. We observe that every vertex $v \in V(K)$ has degree at most $4$. In fact, if $v\notin D$, then $v\in V(H)$ has at most two neighbors in $r'$. Similarly, $v$ and $v'$ have at most two neighbors each in $r'$ if $v\in D$, so that $v$ is the endpoint of at most four edges in $K$ in this case. By König's Lemma, then, there is $r$ a ray in $K$. Since $E(K) = \tau^{-1}(E(r'))$, the ray $\tau(r)$ meets $r'$ in infinitely many edges, so that $[\tau(r)]_E = [r']_E$. Hence, $\Psi$ is surjective.

	In order to conclude that $\Psi$ is injective, let $s$ and $r$ be non-equivalent rays in $G$. In other words, there is a finite set $S\subset V(G)$ that separates $r$ and $s$. Since $[s]\neq [r]$, by our main hypothesis over $G$ no vertex of $S$ edge-dominates both $r$ and $s$. Hence, for each $v \in S$ there is $F_v\subset E(G)$ finite that separates $v$ from $r$ or $s$. We claim that the finite set $\displaystyle \bigcup_{v\in S}\tau(F_v)$ separates $\tau(r)$ and $\tau(s)$ in $H$. For instance, suppose that the tails of $\tau(s)$ and $\tau(r)$ belong to the same connected component $K'$ of $H\setminus \displaystyle\bigcup_{v\in S}\tau(F_v)$. Then, again by Remark \ref{remark33}, there is a path $P$ in $G$ with $\tau(E(P))\subset H\setminus \displaystyle\bigcup_{v\in S}\tau(F_v)$ and connecting $r$ and $s$. However, there is also $v \in V(P)\cap S$ because $S$ separates these two rays. Since $E(P) \cap F_v = \emptyset$, this contradicts the fact that $F_v$ also separates the vertex $v$ from $r$ or $s$. 
\end{proof}

The continuity of $\Psi$ follows easily from the fact that any edge-separator has a natural vertex-separator associated, an argument also employed in Proposition \ref{phi} when we verified that $\Phi$ is an open map:

\begin{prop}
	$\Psi$ is continuous.
\end{prop}
\begin{proof}
	Let $F\subset E(H)$ be a finite set of edges and fix $\eta \in \Omega_E(H)$ an edge-end of $H$. Since $\Psi$ is surjective, there is a ray $r$ in $G$ such that $[\tau(r)]_E = \eta$. For each $e \in F\cap \tau(E(G))$, consider $S_e$ the set of endpoints of $\tau^{-1}(e)$. If $e\in F$ has the form $e = vv'$ for some $v\in D$, define $S_e = \{v\}$. Then, $S = \displaystyle \bigcup_{e\in F}S_e$ is finite.

	Let $C$ be the connected component of $G\setminus S$ in which there is a tail of $r$. If $s$ is any other ray in $C$ and $P$ is a path connecting $r$ and $s$, then $\tau(P)$ connects $\tau(r)$ to $\tau(s)$ in $H\setminus F$ by construction. This argument verifies the inclusion $\Psi(\Omega(S,[r]))\subset \Omega_E(F,\eta)$, proving that $\Psi$ is continuous.
\end{proof}

We observe that the criteria for defining the neighbors of $v$ and $v'$ in $H$, given $v\in D$, was not mentioned in the proofs of the previous propositions. In fact, finishing this section, it is employed only to show that $\Psi$ is an open map:

\begin{prop}
	$\Psi$ is an open map.
\end{prop}
\begin{proof}
	Let $S\subset V(G)$ be finite and fix a connected component $C$ of $G\setminus S$ in which there is a ray $s$. Denote $\varepsilon = [s]$ and recall that we fixed a set $E_{\varepsilon}\subset V(G)$ as in (\ref{envelope}) in order to define $\tau$. Choosing the representative $s$ so that $V(s)\subset E_{\varepsilon}$, we will now show that $\Psi(\Omega(S,[s]))$ is open in $\Omega_E(H)$.

	First, for each $v \in S\setminus E_{\varepsilon}$, we observe that there is $C_v$ a connected component in $G\setminus E_{\varepsilon}$ containing $v$. By Lemma \ref{propriedadesenvelope}, the set $N(C_v) = \{u\in E_{\varepsilon}: u\text{ has a neighbor in }C_v\}$ is finite. Moreover, for each $u\in N(C_v)$, one of the options below is verified:
	\begin{itemize}
		\item If $u\in D$, define the singleton set $F_u = \{uu'\}\subset E(H)$;
		\item If $u\notin D$, there is a finite set $F_u'\subset E(G)$ that separates $u$ from $s$. Hence, we define $F_u = \tau(F_u')$.
	\end{itemize}
	In any case, $F_1 := \displaystyle \bigcup_{v \in S\setminus E_{\varepsilon}}\bigcup_{u \in N(C_v)}F_u$ is a finite set of edges of $H$. Relying on the following claim, a similar set can be defined:
	\begin{center}
		\texttt{Claim:} It is finite the set $$S' = \{v \in E_{\varepsilon}: \text{there is no path connecting }v\text{ and a tail of }s\text{ in }G\setminus S\}.$$  
	\end{center}
	\begin{proof}[Proof of the Claim]
		By definition, we first observe that $S$ separates vertices of $S'\setminus S$ from $s$. Therefore, $D_{\varepsilon}\cap S'\subset S$, where $D_{\varepsilon}$, as in (\ref{envelope}), is the set of vertices of $G$ that dominate $\varepsilon$. If we assume, for a contradiction, that $S'$ is infinite, then $S'\cap \displaystyle \bigcup_{r\in \mathcal{R}_{\varepsilon}}V(r)$ is infinite, where $\mathcal{R}_{\varepsilon}$ is also defined by Lemma \ref{propriedadesenvelope}. If $S'\cap V(r)$ is infinite for some $r \in \mathcal{R}_{\varepsilon}$, then there is a path connecting $s$ to a vertex of $S'\cap V(r)$ in $G\setminus S$, since $[s] = \varepsilon = [r]$. This contradicts the definition of $S'$, so that we may have $S'\cap V(r)$ finite for every $r\in \mathcal{R}_{\varepsilon}$. In particular, $\mathcal{R}_{\varepsilon}' = \{r \in \mathcal{R}_{\varepsilon}: S'\cap V(r) \neq \emptyset\}$ is infinite and composed by pairwise disjoint rays. Hence, there is also some ray $r\in \mathcal{R}_{\varepsilon}'$ such that $S\cap V(r) = \emptyset$, because $S$ is finite. Therefore, due to the equivalence between the rays $r$ and $s$, there is in $G\setminus S$ a path connecting a tail of $s$ to a vertex of $S'\cap V(r)$. This contradicts the definition of $S'$ once more, so that $S'$ must be finite. 
	\end{proof}
	
	If $v \in S'$ edge-dominates $s$, we define the set $J_v = \{vv'\}\subset E(H)$. Otherwise, there is a finite set $J_v'\subset E(G)$ that separates $s$ and $v$. In this case, we denote $J_v = \tau(J_v')$. Then, $F_2 := \displaystyle \bigcup_{v\in S'}J_v$ is a finite set of edges of $H$.

	Setting $F:= F_1\cup F_2$, let $C'$ be the connected component of $H\setminus F$ in which $\tau(s)$ has a tail. If another ray in that component has the form $\tau(s')$ for some ray $s'$ in $G$, then, according to Remark \ref{remark33}, there is a path $P$ in $G$ connecting $s$ and $s'$ such that $\tau(E(P))\subset E(H)\setminus F$.

	On the other hand, if $[s]\neq [s']$, by Lemma \ref{propriedadesenvelope} there is a connected component $C_{s'}$ in $G\setminus E_{\varepsilon}$ in which $s'$ has a tail. Hence, writing $P$ in terms of its vertices as $v_0v_1v_2\dots v_n$, also assuming that $v_0 \in E_{\varepsilon}$ and $v_n \in C_{s'}$, fix $i = \min \{0 \leq j \leq n : v_j \in E_{\omega}\text{ and }v_{j+1}\in C_{s'}\}$. Since $v_i \in E_{\varepsilon}$, one of the following cases must hold:
	\begin{itemize}
		\item If $v_i$ edge-dominates $s$, the edge $v_iv_{i}'$ is defined in $H$. Moreover, by definition of $\tau_{v_i}$, the path $P$ must contain this edge, because $v_{i+1}\in C_{s'}$. Hence, $v_iv_{i}'$ does not belong to $F$. By definition of $F_1$, this means that $S\cap C_{s'} = \emptyset$, while, by definition of $F_2$, $v_i\notin S'$;
		\item Supposing now that $v_i$ can be separated from $s$ by finitely many edges, we have $S\cap C_{s'} = \emptyset$. Otherwise, $E(P)\cap F_{v_i}' \neq \emptyset$, because $P$ connects the tails of $s$ and $s'$ and contains $v_i$, contradicting the fact that $\tau(E(P))\subset H\setminus F_1$. Analogously, if $v_i \in S'$ for instance, then we have $E(P)\cap J_{v_i}' \neq \emptyset$ by the same reason, contradicting the fact that $\tau(E(P))\subset H\setminus F_2$.
	\end{itemize}

	In both cases, we conclude that $S\cap C_{s'} = \emptyset$ and $v_i\notin S'$. Then, $C_{s'}$ is a connected subgraph of $G\setminus S$ containing $v_{i+1}$ and a tail of $s'$, while there is also a path in $G\setminus S$ connecting $v_i$ and a tail of $s$. Therefore, the tails of $s$ and $s'$ belong to the same connected component of $G\setminus S$, proving that $\Omega_E(F, [s]_E)\subset \Psi(\Omega(S,[s]))$. 
\end{proof}

\section{Topological implications of Theorem \ref{volta}}\label{contraexemplo}
\paragraph{}

Although Theorems \ref{volta} and \ref{ida} present a characterization of edge-end spaces in terms of the end spaces as originally defined in the literature, it is not immediate to assess whether these classes of topological spaces are the same. Actually, edge-end spaces define a strictly smaller subfamily of all the end spaces, which this Section aims to conclude in order to finish the proof of Theorem \ref{t1}.

However, this investigation relies on the work of Kurkofka and Pitz in \cite{representacao}, where these authors proved that end-spaces of graphs are precisely the topological spaces arising from ray spaces of special (order) trees. In particular, given a graph $G$, they constructed a tree $T$ such that the end-space of $G$ is also the end-space of some graphs on $T$. As the program carried out by this Section, we will verify that, if $G$ is under the hypothesis of Theorem \ref{ida}, such construction can lead to a countable order tree $T$, assuming convenient topological properties. Since countable graphs have metrizable end-spaces, by admitting normal (graph-theoretic) spanning trees for instance, this will prove the following:

\begin{thm}\label{t2}
	Let $X$ be a Lindelöf and first-countable topological space. If $X$ is the edge-end space of some graph, then $X$ is metrizable. 
\end{thm}

Introducing the main notation of this section, we recall that a \textbf{tree}, in a set-theoretic context, is a partially ordered set $\langle T, \leq \rangle$ where $\mathring{\lceil t\rceil} = \{s \in T : s < t\}$ is well-ordered by $\leq$ for every \textbf{node} $t \in T$. The order type $\mathrm{h}(t)$ of $\mathring{\lceil t\rceil}$ is the \textbf{height} of the node $t\in T$, so that, for a given ordinal $\alpha$, the set $\mathcal{L}_{\alpha}(T) = \{t \in T : t\text{ has height }\alpha\}$ is the $\alpha$ \textbf{level} of $T$. If there is a \textbf{predecessor} $s = \max \mathring{\lceil t\rceil}$, we say that $t\in T$ is a \textbf{successor} point. Otherwise, the cofinality of $\mathring{\lceil t\rceil}$ is infinite and we say that $t$ is a \textbf{limit} point of $T$. For the least ordinal $\alpha$ such that $\mathcal{L}_{\alpha}(T) = \emptyset$, we say that $\alpha$ is the height of $T$.

Throughout this paper, trees will always be \textbf{rooted}, i.e., such that $|\mathcal{L}_0(T)| = 1$. Moreover, a set $R\subset T$ in which the order $\leq$ is total is called a \textbf{chain} of $T$, while a maximal one is called a \textbf{branch}. If $T$ has no infinite branches, we say that $T$ is \textbf{rayless}. Following the notation of \cite{representacao}, we say that $R\subset T$ is a \textbf{high-ray} if it is a down-closed chain of cofinality $\omega$. Then, the \textbf{tops} of this high-ray are the minimal elements greater then every node of $R$.

Graph-theoretic trees, i.e., the acyclic connected graphs, are examples of trees whose height is bounded by $\omega$, if endowed with the usual tree-order after fixing a root. On the other hand, Brochet and Diestel in \cite{tgrafos} introduced canonical graphs associated to a fixed order tree $T$. Following their terms, a $T-$\textbf{graph} is a graph whose vertex set is $T$ and such that, for every $t\in T$, its neighborhood $N(t)$ is cofinal in $\mathring{\lceil t\rceil}$. Moreover, if, for every limit node $t\in T$, there is a finite subset $X\subset \mathring{\lceil r\rceil}$ such that every $s > t$ has its neighbors below $t$ as elements of $X$, we say that $G$ is a \textbf{uniform} $T-$graph.

In particular, the endpoints of any edge in a $T-$graph are comparable in the order of $T$. Therefore, this definition generalizes the concept of \textit{normal spanning trees}, whose first studies by Jung for infinite graphs (see \cite{jung} for example) provided useful tools for studying topological and combinatorial behaviour of ends. In this context, we recall that, for a given graph $G$ and a fixed subgraph $T$, a $T -$\textbf{path} in $G$ is a path that has precisely its endpoints in $H$. If $T$ is an order tree, we say that $T$ is \textbf{normal} if every $T-$path has comparable endpoints in the tree order.

Both in finite and infinite settings, normal trees in graphs are usually found by \textit{depth-search algorithms}. For quite broad graph classes in the literature, such as for countable ones, these algorithms lead even to normal spanning trees. As an example to be recovered further in this Section, we outline the unified method developed by Pitz in \cite{buscaprofundidade} to obtain these structures, while revisiting his original proof in order to set a convenient statement for our purposes:  

\begin{prop}[\cite{buscaprofundidade}, Theorem 3]\label{buscaprofundidade}
	Let $G$ be any graph and fix a finite subset $K\subset V(G)$. Write $G'$ for the graph $ G\setminus K$. Then, $G'$ has a maximal normal (graph-theoretic) tree $T$ such that, for every connected component $C$ of $G'\setminus T$:
	\begin{itemize}
		\item The neighborhood $N(C) = \{v \in G\setminus C : v \text{ has a neighbor in }C \}$ is infinite. In particular, $N(C)\cap T$ is contained in a infinite branch $r_C$ of $T$;
		\item Every $v \in N(C)$ dominates $r_C$.
	\end{itemize}
\end{prop}

\begin{proof}[Revisited proof from \cite{buscaprofundidade}]
	We will construct an increasing sequence $T_0\subseteq T_1\subseteq T_2\subseteq T_3\subseteq \dots$ of rayless normal trees in $G'$, by first fixing $T_0 = \{r\}$ an arbitrary root. Suppose that $T_n$ is defined for some $n <\omega$. Let $\mathcal{C}_n$ denote the collection of the connected components of $G'\setminus T_n$. Since $T_n$ is rayless and normal, the neighborhood $N(C_n) = \{v \in G\setminus C : v \text{ has a neighbor }x\in C_n\}$ is finite for every $C_n\in \mathcal{C}_n$, because $K$ is finite and $N(C)\cap T_n$ is contained in a branch of $T$.

	For every vertex $v\in N(C_n)$, fix $x_v^n\in C_n$ one of its neighbors. Then, by the previous observation, $F_{C_n} = \{x_v^n : v \in N(C_n)\}$ is finite. Moreover, if $r$ is a ray in $G$ that has a tail in $C_n \in \mathcal{C}_n$, then $N(C_n)\cup F_{C_n}$ is a finite set of vertices that separates $r$ from any vertex of $D_n := \displaystyle \bigcup_{C_n\in \mathcal{C}_n}F_{C_n}$. Due to this property, we say that $D_n$ is \textit{disperse}. Hence, by a well-known result of Jung in \cite{jung} (see also Theorem 2.2 of \cite{paracompacidade}), there is\footnote{For a more immediate construction of $T_{n+1}$, we may apply a routine depth-search algorithms within each connected component $C_n\in \mathcal{C}_n$, reaching the finitely many vertices of $F_{C_n}$. See Proposition 1.5.6 of \cite{livrodiestel} for details.} $T_{n+1}$ a rayless normal tree in $G'$ extending $T_n$ and containing $D_n$. Moreover, $T_{n+1}\cap C_n$ is connected for every $C_n\in \mathcal{C}$. After all, a path in $C_n$ connecting two vertices $v,u\in T_{n+1}\cap C_n$ must intersect $T_{n+1}$ in a third vertex if $u$ and $v$ are not comparable in its tree order, because $C_n\subset G'\setminus T_n$ and $T_{n+1}$ is normal.

	Once this inductive process has finished, we claim that $T = \displaystyle\bigcup_{n < \omega}T_n$ satisfies the statement. We first observe that $T$ is a normal tree, because so is $T_n$ for each $n < \omega$. Next, for a contradiction, suppose that $N(C)$ is finite for some $C \in \mathcal{C}$, where $\mathcal{C}$ denotes the family of connected components of $G'\setminus T$. Therefore, we must have $N(C)\cap T \subset T_n$ for some $n < \omega$. In this case, $C$ is also a connected component of $G'\setminus T_n$, contradicting the fact that $F_C \subset T_{n+1}\setminus T_n$. Hence, the first item of the Proposition holds.

	Now, consider $T'\supsetneq T$ a (graph-theoretic) tree in $G'$ extending the tree order of $T$, if it exists. Let $v \in T'\setminus T$ be minimal, i.e., such that $\mathring{\lceil v\rceil}\subset T$. Note that $\mathring{\lceil v\rceil}$ is finite, due to the fact that $T'$ has height bounded by $\omega$. Then, in $G'\setminus\mathring{\lceil v\rceil}$ there is a path connecting $v$ to a vertex of $T$, because the connected component $C\in \mathcal{C}$ containing $v$ has infinitely many neighbors in $T$. Since $T'$ extends the tree order of $T$, this verifies that $T'$ is not normal. In other words, $T$ is a maximal normal (graph-theoretic) tree.

	Finally, fixed $C\in \mathcal{C}$ and $v\in N(C)$, denote by $r_C$ the branch of $T$ containing $N(C)$. For a finite set $S\subset V(G)\setminus \{v\}$, let $n < \omega$ be big enough so that $v \in T_n$ and $S\cap T = S\cap T_n$. Consider $C_n \in \mathcal{C}_n$ the connected component of $G'\setminus T_n$ in which $C\subset C_n$ and, hence, $v\in N(C_n)$. In $T_{n+1}$, then, there is a path connecting $x_v^n$, a neighbor of $v$ in $C_n$, to a vertex $v'$ of $r_C$, because $T_{n+1}\cap C_n\subset G'\setminus S$ is connected. Therefore, $v$ dominates the ray $r_C$. 
\end{proof}

Relying on the above Proposition, the proof of Theorem \ref{t2} is done by revisiting Theorem 2 of Kurkofka and Pitz in \cite{representacao}. For a given graph $G$, their result is a decomposition of $V(G)$ in the shape of an order tree $T$, codifying $\Omega(G)$ through the high-rays of $T$. More precisely, a pair $(T,\mathcal{V})$ is called a \textbf{partition tree} of $G$ if $\mathcal{V} = \{V_t : t\in T\}$ is a partition of $V(G)$ into connected subsets satisfying the properties below:

\begin{itemize}
	\item $|V_t| = 1$ if $t\in T$ is not a limit point;
	\item The graph $\dot{G} := \frac{G}{\mathcal{V}}$ obtained by contracting the parts of $\mathcal{V}$ to single vertices is a $T-$graph;
	\item For each successor $t\in T$, the neighborhood $$N(V_{\lfloor t\rfloor}) = \left\{u \in V(G)\setminus \displaystyle V_{\lfloor t\rfloor} : u \text{ has a neighbor in }\displaystyle V_{\lfloor t\rfloor}\right\}$$ is finite, where $V_{\lfloor t\rfloor} = \displaystyle \bigcup_{s \geq t}V_s$. In this case, we say that $(T,\mathcal{V})$ has \textbf{finite adhesion}.
\end{itemize}

Hence, by tracking some rays of $G$, we are able to describe high-rays of $T$. In other words, given $[r]\in \Omega(G)$ an end, it is well-defined the set 
\begin{equation}\label{theta}
	\Theta([r]) = \{t \in T : r \text{ has a tail in }\displaystyle V_{\lfloor t\rfloor}\}.
\end{equation} This is clearly a down-closed chain of $T$ and, by Lemma 6.2 of \cite{representacao}, it has countable cofinality. If this cofinality is infinite, we say that $\Theta([r])$ \textbf{corresponds} to the end $[r]$, because $\Theta([r])$ is then an element of $\mathcal{R}(T) = \{\text{high-rays of }T\}$. When every $[r]\in \Omega(G)$ corresponds to precisely one high-ray of $T$, in the sense that $\Theta : \Omega(G) \to \mathcal{R}(T)$ is a bijection, we even say that the partition tree $(T,\mathcal{V})$ \textbf{displays} all the ends of $G$. The existence of partition trees with that property is guaranteed by Theorem 2 of \cite{representacao}, whose proof is partially adapted below to provide a convenient statement for the graph classes we aim to approach:

\begin{thm}[\cite{representacao}, Theorem 2] \label{Tchapeu}
	Let $G$ be a connected graph whose vertices dominate at most one end. Suppose that $\Omega(G)$ is a first-countable Lindelöf topological space. Then, $G$ has a partition tree $(T,\mathcal{V})$ that displays all its ends and such that:
	\begin{enumerate}
		\item $T$ has countable height, bounded by $\omega\cdot \omega$;
		\item The subtree $\hat{T} = \{t\in T : t\text{ belongs to a high-ray of }T\}$ is countable.
	\end{enumerate}
\end{thm}
\begin{proof}[Revisited proof from \cite{representacao}]
	For some cardinal $\kappa \leq \omega$, we will recursively construct a sequence of partition trees $\{(T_n,\mathcal{V}_n)\}_{n < \kappa}$ for $G$.
	Let $T_0'$ be a maximal normal tree for the graph as in Proposition \ref{buscaprofundidade}, whose tree-order is denoted by $\leq$. Consider $F_0$ the set of connected components of $G\setminus T_0'$, so that each $C\in F_0$ has its infinite neighborhood $N(C)$ contained in a branch $r_C$ of $T_0'$. Then, the order $\leq$ can be extended to $T_0: = T_0'\cup F_0$, by declaring $C > t$ for every $C\in F_0$ and every $t\in r_C$. If we set $V_t^0= \{t\}$ for every $t\in T_0'$ and $V_C^0 = V(C)$ for every $C\in F_0$, we partition $G$ with the family $\mathcal{V}_0 = \{V_t^0 : t \in T_0'\}\cup \{V_C^0 : C \in F_0\}$. It is easily verified that $(T_0,\mathcal{V}_0)$ is a partition tree for $G$.

	Now, for some $n<\omega$, suppose that we have defined a partition tree $(T_n, \mathcal{V}_n)$ of height bounded by $\omega\cdot (n+1)+1$. If the height of $T_n$ is precisely $\omega \cdot (n+1)$, we set $\kappa = n+1$ and finish the construction. Otherwise, by induction, the nodes of height $\omega\cdot n$ in $T_n$ define a set $F_n$ with the following form:
	
	\begin{equation}\label{limites}
		C \in F_n \iff C\text{ is a connected component of }G\setminus \bigcup_{t\in T_n\setminus F_n}V_t.
	\end{equation}

	Moreover, for each $C\in F_n$, we assume that its neighborhood $N(C) = \{v \in V(G)\setminus C : v\text{ has a neighbor in }C\}$ is contained in a ray $r_C$ such that $\Theta([r_C])$ is a high-ray of $T_n$. Hence, for every $C \in F_n$ we can apply Lemma 7.2 of \cite{representacao} to obtain a connected vertex set $U_C\subseteq V(C)$ that encodes suitable topological properties. Among these, we mention:
	\begin{center}
		\texttt{Fact:} If $\mathcal{D}(C)$ is the set of connected components of $C\setminus U_C$, the subgraph $D\in \mathcal{D}(C)$ has finite neighborhood in $G$. In other words, the set $N(D) = \{v \in G\setminus D : v\text{ has a neighbor in }D\}$ is finite.
	\end{center}

	Then, rooted at a vertex that has some neighbor in $U_C$, we can fix $T_D$ a normal tree for $D$, obtained when applying Proposition \ref{buscaprofundidade} to $G[D\cup N(D)]$ with $K = N(D)$. Hence, every connected component $D'$ of $D\setminus T_D$ has infinitely many neighbors within a branch $r_{D'}$ of $T_D$. Moreover, since $N(D)$ is finite and $D'\subseteq D$, we have that $N(D') \setminus D$ is finite. Then, writing $ \mathcal{V}_n = \{V_t^n\}_{t\in T_n}$ and $\mathcal{V}_{n+1} = \{V_t^{n+1}\}_{t\in T_{n+1}}$, the partition tree $(T_{n+1},\mathcal{V}_{n+1})$ can be described as follows:
	\begin{itemize}
		\item The tree $T_{n+1}$ extends the order tree $T_n$, by additionally containing the nodes from $T_D$ for every $C\in F_n$ and every $D\in \mathcal{D}(C)$. In this case, we set $t> C$ for each $t \in T_D$. Moreover, for every connected component $D'$ of $D\setminus T_D$, we also see $D'$ as a node of $T_{n+1}$, defining $D' > t$ for every $t \in r_{D'}$;
		\item We set $V_t^{n+1} = V_t^n$ for every $t\in T_n\setminus F_n$. Given $C\in F_n$ and $D\in \mathcal{D}(C)$, however, we define $V_C^{n+1} = U_C$ and $V_t^{n+1} = \{t\}$ for every $t\in T_D$. Finally, we set $V_{D'}^{n+1} = V(D')$ for every connected component $D'$ of $D\setminus T_D$. 
	\end{itemize}

	Then, $|V_t^{n+1}| = 1$ for every successor $t \in T_{n+1}$ and $\frac{G}{\mathcal{V}_{n+1}}$ is indeed a $T_{n+1}-$graph. Relying on the above \texttt{Fact}, the choice of $V_{t}^{n+1}$ for a limit node $t\in F_n$ guarantees that $(T_{n+1},\mathcal{V}_{n+1})$ has finite adhesion. If $(T_n,\mathcal{V}_n)$ is defined for every $n < \omega$, we set $\kappa = \omega$. Then, writing $\mathcal{V} = \{V_t\}_{t\in T}$, the requested partition tree $(T, \mathcal{V})$ arises from the following limit definition:
	\begin{itemize}
		\item We set $T = \displaystyle \bigcup_{n < \kappa}T_n$, extending the order of $T_n$ for every $n < \kappa$;
		\item If $\kappa = n+1 < \omega$, we consider $\mathcal{V} = \mathcal{V}_n$, so that $(T,\mathcal{V})$ is the partition tree $(T_n,\mathcal{V}_n)$. Otherwise, given $t\in T$, we set $V_t = V_t^n$ for any $n > \min\{i < \omega : t\in T_i\}$. In this case, we claim that $\mathcal{V} = \{V_t\}_{t\in T}$ is a partition of $V(G)$. For instance, suppose that there is a vertex $v\in V(G)\setminus \displaystyle \bigcup_{t\in T}V_t$. As $(T_n,\mathcal{V}_n)$ is a partition tree for each $n<\omega$, we must have $v \in C_n$ for some $C_n \in F_n$. By considering a big enough $n_0 < \omega$, let $u \in \displaystyle \bigcup_{t \in T_{n_0}}V_t^{n_0}$ be a neighbor of $v$.  By the second item of Proposition \ref{buscaprofundidade}, $u$ dominates the ray $r_{C_n}$ for every $n \geq n_0$. However, if $n > m \geq n_0$, the choice of $U_{C_{m}}$ guarantees that $r_{C_n}$ and $r_{C_m}$ can be separated by finitely many vertices. In other words, $u$ dominates infinitely many non-equivalent rays, which is a contradiction. Hence, $\mathcal{V}$ is indeed a partition of $G$.
	\end{itemize}

	The verification that $T$ displays the ends of $G$ is precisely the one given by \cite{representacao}. Since $T$ has countable height, bounded by $\omega \cdot \omega$, it remains to show that the subtree $\hat{T} = \{t\in T : t\text{ belongs to a high-ray of }T\}$ is itself countable.

	For instance, suppose that $\hat{T}$ is uncountable and fix $$\alpha = \min \{\xi < \omega \cdot \omega : \text{the }\xi\text{ level of }\hat{T}\text{ is uncountable}\}.$$ That ordinal exists because $\hat{T}$ has countable height, and we have $\alpha > 0$ since $\hat{T}$ is rooted. If $\alpha$ is a successor ordinal, written as $\alpha = \beta + 1$, let $t\in \mathcal{L}_{\beta}(\hat{T})$ be a predecessor of an uncountable family of nodes $\{t_{i}\}_{i < \omega_1}\subset \mathcal{L}_{\alpha}(\hat{T})$. Also relying on the previous \texttt{Fact} and by passing $\{t_i\}_{i < \omega_1}$ to another uncountable subsequence if necessary, we can assume that $N(V_{\lfloor t_i\rfloor}) = N(V_{\lfloor t_j\rfloor}) = : S$ for every $i,j < \omega_1$, since $h(t)$ is countable and $(T,\mathcal{V})$ has finite adhesion. Then, $\{\Omega(S,\varepsilon) : \varepsilon \in \Omega(G)\}$ is an open cover for $\Omega(G)$ whose distinct elements are disjoint. However, by definition of $\hat{T}$, there is $r_i$ a ray in $G$ that has a tail in $V_{\lfloor t_i\rfloor}$, for each $i < \omega_1$. Hence, $S$ separates $r_i$ and $r_j$ if $i \neq j$. This means that $\{\Omega(S,[r_i]) : i < \omega_1\}\subset \{\Omega(S,\varepsilon) : \varepsilon \in \Omega(G)\}$
	is an uncountable subfamily whose elements are pairwise disjoint, contradicting the assumption that $\Omega(G)$ is a Lindelöf topological space.

	Therefore, $\alpha$ must be a limit ordinal, so that $\mathcal{L}_{\alpha}(T) = F_n$ for some $n < \omega$. We argue that $n \neq 0$. Otherwise, fix $\{C_i\}_{i < \omega_1}\subset F_0 \cap \hat{T}$ uncountable. Recall that these are connected components of $G\setminus T_0'$, and, therefore, are pairwise disjoint. Then, one of the following cases is verified, but both lead to contradictions:
	
	\begin{itemize}
		\item If there is an uncountable subset $I\subset \omega_1$ such that $r_{C_i} = r_{C_j} =:r$ for every $i,j \in I$, then each $C_i$ has infinitely many neighbors in the branch $r$ of $T_0'$.  As before, for each $i \in I$, let $r_i$ be a ray in $V_{\lfloor C_i\rfloor}$, whose existence is guaranteed by the definition of $\hat{T}$ and by the fact that $(T,\mathcal{V})$ displays the ends of $G$. Then, given a finite subset $S\subset V(G)$, we have $[r_i]\in \Omega(S,[r])$ for all but finitely many indices $i \in I$. This, however, contradicts the fact that $\Omega(G)$ is a first-countable topological space; 
		\item Then, there is an uncountable subset $I\subset \omega_1$ such that $r_{C_i}\neq r_{C_j}$ for every $i,j \in I$. In this case, $[r_{C_i}]\neq [r_{C_j}]$, because $(T,\mathcal{V})$ displays the ends of $G$. Since $r_{C_i}$ is a branch of $T_0'$, we have $r_{C_i}\subset \hat{T}$ for each $i \in I$. Fix $v_i\in r_{C_i}$ a neighbor of the connected component $C_i$. As $\hat{T}\cap T_0'$ is countable by the minimality of $\alpha$, there must be $v \in T_0'$ such that $v = v_i$ for uncountably many indices $i\in I$. According to Proposition \ref{buscaprofundidade}, this means that $v$ dominates uncountably many non-equivalent rays, contradicting the main hypothesis over $G$. 
	\end{itemize}

	Then, we must have $n > 0$. Moreover, $F_{n-1}\cap \hat{T}$ is countable, since this is a smaller (limit) level of $\hat{T}$. Hence, for some $C\in F_{n-1}\cap \hat{T}$, the set $\{t \in F_n \cap \hat{T}: t > C\}$ is uncountable. As an element of $\hat{T}$, the node $C$ has only countably many successors, because $\alpha$ is a limit ordinal. Therefore, we can fix $v_0\in \hat{T}$ a successor of $C$ such that $\{t\in F_n\cap \hat{T}: t > v_0\}$ is uncountable. By construction, we recall that $v_0$ is the root of a normal tree $T_D$ for a connected component $D$ of $C\setminus V_C^t$. Fixing an uncountable family $\{D_i'\}_{i < \omega_1}\subset \{t\in F_n\cap \hat{T}: t > v_0\}$, then, each $D_i'$ is a connected component of $D\setminus T_D$. Analogously to the above discussion, one of the following cases is verified, but both also lead to contradictions:
	
	\begin{itemize}
		\item Suppose that there is an uncountable set $I\subset \omega_1$ such that $r_{D_i'} = r_{D_j'}=:r$ for every $i,j\in I$. Since $\{D_i'\}_{i<\omega_1}\in F_n \cap \hat{T}$, the branch $r$ of $T_D$ is contained in $\hat{T}$. For each $i\in I$, let $r_i$ be a ray in $V_{\lfloor C_i\rfloor}$, whose existence is guaranteed by the definition of $\hat{T}$ and by the fact that $(T,\mathcal{V})$ displays the ends of $G$. Then, given a finite subset $S\subset V(G)$, we have $[r_i]\in \Omega(S,[r])$ for all but finitely many indices $i \in I$. This, however, contradicts the fact that $\Omega(G)$ is a first-countable topological space; 
		\item Then, there is an uncountable set $I\subset \omega_1$ such that $r_{D_i'} \neq r_{D_j'}$ for every $i,j\in I$. In this case, $[r_{D_i'}]\neq [r_{D_j'}]$, because $(T,\mathcal{V})$ displays the ends of $G$. Since $r_{D_i'}$ is a branch of $T_{D}$ for every $i\in I$, we have $r_{D_i'}\subset \hat{T}$, allowing us to choose $v_i\in \hat{T}\cap T_D$ a neighbor of $D_i'$. Observing that $\hat{T}\cap T_D$ is countable by the minimality of $\alpha$, there is $v\in \hat{T}\cap T_D$ a vertex such that $v = v_i$ for every $i$ within some uncountable subset of $I$. By Proposition \ref{buscaprofundidade}, this means that $v$ dominates uncountably many distinct ends of $G$, contradicting our main hypothesis over this graph.
	\end{itemize}

	Therefore, $\hat{T}$ is countable.
	
\end{proof}

Actually, by the proof of Theorem 2 in \cite{representacao}, one can translate the convergence of ends in $G$ through combinatorial properties of the partition tree $(T,\mathcal{V})$ constructed above. As discussed in that paper, this is a \textit{sequentially faithful} partition tree displaying all the ends of $G$, which encodes the topology of $\Omega(G)$ through the tree-structure of $T$. In our settings, Theorem 1 of \cite{representacao} proves the following:  

\begin{prop}[\cite{representacao}, Theorem 1]\label{tgrafo}
	Let $G$ be a graph as in Theorem \ref{Tchapeu}. Then, there is $T'$ an order tree such that $\Omega(G)$ is the end space of any uniform $T'-$graph. Moreover, $T'$ can be chosen so that $\hat{T'} = \{t\in T' : t \text{ belongs to a high-ray of }T'\}$ is countable.
\end{prop}
\begin{proof}[Revisited proof from \cite{representacao}]
	Fix the sequentially faithful partition tree $(T,\mathcal{V})$ for $G$ given by Theorem \ref{Tchapeu}. Let $T'$ be obtained from $T$ according to the procedure below:
	\begin{enumerate}
		\item For every limit node $t\in T$, denote by $S(t)$ the set of its successors, if there are some. Then, since $(T,\mathcal{V})$ has finite adhesion, the set $N_s : = N(V_{\lfloor s \rfloor}) $ is finite for each $s\in S(t)$; 
		\item Now, for every limit node $t\in T$ that has a successor and every finite $X\subset \lceil t \rceil$, we declare a new node $v(t,X)$ to be a successor of $t$ and a predecessor of each $s\in S(t)$ with $N_s = X$. We then remove $t$.
	\end{enumerate}

	Since only non-empty levels of $T$ were modified to construct $T'$, the height of this latter tree is also countable. In addition, the proof that $\Omega(G)$ is the end space of any $T'-$graph is precisely the one given in \cite{representacao}.

	Finally, let $t'\in T'$ be a node that belongs to a high-ray of $T'$. If $t'$ is a successor node in $T'$, then it is also a successor node in $T$ by construction of $T'$. In particular, $t'$ also belongs to a high-ray of $T$, so that $t'\in \hat{T}$. If $t'$ is a limit point, however, then $t' = v(t, X)$ for some $t\in T$ that has at least one successor and some finite subset $X\subset \lceil t \rceil$. Actually, the node $t$ must lie on a high-ray of $T$ (i.e., $t\in \hat{T}$), because $t'$ itself belongs to a high-ray of $T'$. Hence, since $\hat{T}$ is countable and there are countably many finite subsets of $\lceil t \rceil$ for every $t\in \hat{T}$, it follows that $\hat{T'}$ is also countable.   
\end{proof}

In its original statement, Theorem 1 of \cite{representacao} proves that every end space is the end space of some special order tree. For completeness, we recall that an order tree $T$ is \textbf{special} if it is a countable disjoint union of \textit{antichains}, i.e., sets of pairwise incomparable elements. Then, there are two different ways to see $\mathcal{R}(T)$ as a topological space, called the \textbf{ray space of the tree} $T$. Formalized by Lemma 5.3 of \cite{representacao}, the first way considers $\mathcal{R}(T)$ as the end space of any uniform $T-$graph $G$, identifying $\mathcal{R}(T)$ and $\Omega(G)$ through the map $\Theta$ as in (\ref{theta}) after considering the partition tree $(T, \{V_t\}_{t\in T})$. The second way, introduced by Lemma 2.1 of \cite{caracterizacao}, is defined intrinsically in terms of $T$, by declaring as basic open neighborhood around the high-ray $r\in \mathcal{R}(T)$ a set of the form \begin{equation} \label{intrinseco}
	[t,F] = \{s \in \mathcal{R}(T): t\in s\text{ and }t'\notin s\text{ for every }t'\in F\},
\end{equation} where $t\in r$ and $F\subset T$ is a finite collection of tops of $r$. The equivalence between these two approaches is discussed by Max Pitz in subsection 2.4 of \cite{caracterizacao}, as well as in Proposition 5.5 of \cite{representacao}. However, relying on both notions, we can conclude the main result of this section:

\begin{thm}
	Let $G$ be a graph in which every vertex dominates at most one end. If $\Omega(G)$ is a Lindelöf first-countable topological space, then $\Omega(G)$ is metrizable.  
\end{thm}
\begin{proof}
	According to Proposition \ref{tgrafo}, there is $T$ an order tree such that $\Omega(G)$ is the end space of any uniform $T-$graph. Moreover, $T$ can be chosen so that $\hat{T} = \{t\in T : t\text{ belongs to a high-ray of }T\}$ is countable. However, $\mathcal{R}(T)$ and $\mathcal{R}(\hat{T})$ describe the same topological space, since the basic open neighborhoods given by (\ref{intrinseco}) are coincident. In particular, if $G'$ is a uniform $\hat{T}-$graph, its end space is homeomorphic to $\Omega(G)\simeq \mathcal{R}(T)$. However, $G'$ is countable, because so is $\hat{T}$. Therefore, $\Omega(G')$ is metrizable, since $G'$ has a normal spanning tree (see \cite{spanningtrees}). 
\end{proof}

Combining Theorem \ref{volta} with this statement, Theorem \ref{t2} follows. As a consequence, we are ready to exhibit a topological space in $\Omega\setminus \Omega_E$, also finishing the proof of Theorem \ref{t1}:

\begin{corol}
	The Alexandroff duplicate of the Cantor set is not an edge-end space, although it is the end space of a graph.
\end{corol}
\begin{proof}
	  The Alexandroff duplicate of the Cantor space is known to be homeomorphic to the end space of a suitable special tree $T$, as first highlighted by Kurkofka and Melcher in their Example 4.1 from \cite{KurkofkaMelcher} and later also discussed by Pitz through his Example 2.6 in \cite{caracterizacao}. Following these references, consider $T$ as obtained from the binary tree $2^{<\omega}$ after adding, for each branch $r\subset 2^{<\omega}$, a top $t_r$ and a countable increasing sequence of nodes $\{t_r = t_r^0 < t_r^1 < t_r^2 < t_r^3 < \dots \}$. We observe that $T$ is special, since it has height $\omega \cdot 2$ and its levels are antichains. Moreover, the space $\mathcal{R}(T)$ is first-countable, because the family of basic open neighborhoods described by (\ref{intrinseco}) is countable. In particular, for each branch $r$ of $2^{<\omega}$, the high-ray $\overline{r} = r\cup \{t_r^i : i <\omega\}$ is an isolated point of $\mathcal{R}$. By Proposition 2.16 of \cite{caracterizacao}, this is also a Lindelöf (actually compact) space, once the nodes of $T$ have countably many successors. Nevertheless, $\mathcal{R}(T)$ is not second-countable, because $\{\{\overline{r}\} : r \text{ is a branch of }2^{<\omega}\}$ is an uncountable family of pairwise disjoint open sets. Hence, $\mathcal{R}(T)$ is not a metric space. Therefore, in the graph-theoretical setting, Theorem \ref{t2} implies that the end space of any uniform $T-$graph\footnote{In this case, a $T-$graph can be obtained by declaring successor nodes to be adjacent to their predecessors, while defining the edges between each top $t_r$ and the nodes of the corresponding branch $r\subset 2^{<\omega}$. For constructions of graphs on arbitrary special trees, see Theorem 4.6 of \cite{representacao}.  } is not an edge-end space.
\end{proof}

\section*{Acknowledgments}
\paragraph{}
The first and the third named authors thank the support of Fundação de Amparo à Pesquisa
do Estado de São Paulo (FAPESP), being sponsored through grant numbers 2023/00595-6 and
2021/13373-6 respectively. In its turn, the second named author acknowledges the support of
Conselho Nacional de Desenvolvimento Científico e Tecnológico (CNPq) through grant number
165761/2021-0.

The authors also thank the anonymous referee for their detailed revisions and the suggested improvements. We are particularly grateful to the reviewer that observed how Remark \ref{remark33} could simplify the discussions carried out within Section \ref{sec:construction}.

\bibliography{referencias}
\bibliographystyle{amsplain}

\end{document}